\documentclass[11pt]{amsart}
\usepackage{latexsym}
\usepackage{amssymb,amsmath,mathabx, amscd}
\usepackage[alphabetic,backrefs,lite]{amsrefs}
\usepackage[pdftex]{graphicx}
\usepackage{enumerate}
\usepackage{tikz}
\usetikzlibrary{decorations.pathreplacing}
\usepackage[margin=1in]{geometry}
\usepackage{courier}
\usepackage{tikz-cd}
\usepackage{color}
\usepackage{colonequals}
\usepackage{quoting} %

\usepackage{MnSymbol}
\usepackage{hyperref}
\usepackage{mathrsfs, colonequals}

\newcommand{\rr}{\mathbb{R}}
\newcommand{\pp}{\mathbb{P}}

\newcommand{\qq}{\mathbb{Q}}

\newcommand{\ff}{\mathbb{F}}

\newcommand{\PP}{\mathbb{P}}

\renewcommand{\O}{\mathcal{O}}

\renewcommand{\P}{\mathcal{P}}

\newcommand{\fa}{\mathfrak{a}}
\newcommand{\fp}{\mathfrak{p}}

\newcommand{\ord}{\text{ord}}

\newcommand{\Pic}{\operatorname{Pic}}

\newcommand{\Aut}{\operatorname{Aut}}

\renewcommand{\epsilon}{\varepsilon}
\newcommand{\Fix}{\operatorname{Fix}}

\newtheorem{thm}{Theorem}[section]

\newtheorem{ithm}{Theorem}
\newtheorem{icor}[ithm]{Corollary}
\newtheorem{lem}[thm]{Lemma}
\newtheorem{conj}[thm]{Conjecture}

\newcommand{\defi}[1]{\textsf{#1}}

\theoremstyle{definition}

\newtheorem{example}[thm]{Example}

\newtheorem{iquestion}[ithm]{Question}

\theoremstyle{remark}
\newtheorem{rem}[thm]{Remark}

\usetikzlibrary{arrows.meta}

\newcommand*{\TikzArrow}[1][]{\mathbin{\tikz [baseline=-0.25ex,-latex,#1] \draw [#1] (0pt,0.5ex) -- (1.3em,0.5ex);}}%

\newcommand{\rat}{\TikzArrow[->, densely dashed]}

\newcommand{\embed}{\hookrightarrow}

\newcommand{\wt}[1]{\widetilde{#1}}
\newcommand{\cO}{\mathcal{O}}
\newcommand{\ot}{\otimes}

\title{Quadratic points on double planes}

\author{Nathan Chen}
\address{Department of Mathematics, Harvard University, 1 Oxford Street, Cambridge MA 02138}
\email{nathanchen@math.harvard.edu}
\urladdr{https://nathanchenmath.github.io}

\author{Ben Church}
\address{Department of Mathematics, Stanford University, 450 Jane Stanford Way Building 380, Stanford, CA 94305, USA}
\email{bvchurch@stanford.edu}
\urladdr{https://web.stanford.edu/~bvchurch/writing/}

\author{Hector Pasten}
\address{Departamento de Matem\'{a}ticas, Pontificia Universidad Cat\'{o}lica de Chile. Facultad de Matem\'{a}ticas,
4860 Av. Vicu\~{n}a Mackenna, Macul, RM, Chile}
\email{hector.pasten@uc.cl}
\urladdr{https://www.mat.uc.cl/~hector.pasten/}

\author{Isabel Vogt}
\address{Department of Mathematics, Brown University, Box 1917, 151 Thayer Street, Providence, RI 02912, USA}
\email{ivogt.math@gmail.com}
\urladdr{https://www.math.brown.edu/ivogt/}

\begin{document}

\begin{abstract}
    Zariski dense collections of quadratic points on curves \(X\) are well-understood by results of Harris--Silverman and Vojta, but when \(\dim X \geq 2\) there is not an analogous geometric characterization, even conjecturally.  In this note we consider the case of a double cover \(\pi \colon X \to \pp^r\), where Hilbert's Irreducibility Theorem implies that the quadratic points in the fibers of \(\pi\) are dense.  We show that Vojta's Conjecture implies that, once the canonical bundle of \(X\) is sufficiently positive, there are no other sources of Zariski dense quadratic points.  This is complemented by several examples of surfaces \(X \to \pp^2\) with an additional source of dense quadratic points. 
\end{abstract}

\maketitle

\section{Introduction}

Let \(k\) be a number field and let \(X/k\) be a smooth projective variety. The algebraic points on \(X\) -- that is the closed points \(x \in X\) together with their residue field \(k(x)\) --- encode the arithmetic of \(X\).  A fruitful guiding principle is that the geometry of \(X\) governs its arithmetic, and hence we expect the geometry of \(X\) to meaningfully impact the behavior of the algebraic points.  For example, the most famous such prediction is the Bombieri--Lang Conjecture, which asserts that a variety of general type does not have Zariski dense \(k\)-points. This expectation is backed up by many theorems in the case that \(\dim X = 1\), including Faltings' Theorem, which settles the dimension \(1\) case of the Bombieri--Lang Conjecture \cite{Faltings-Mordell}.

Some of these results extend to higher degree points.  For example, Harris and Silverman showed that if \(X\) has infinitely many \defi{quadratic points} (closed points \(x\) for which \([k(x):k] = 2\)) then \(X\) is a double cover of a \(\pp^1\) or a positive rank elliptic curve \cite{HS-Degree2} (see also \cite[Theorem 1.2(1)]{KV}).  

While a hyperelliptic curve \(\pi \colon X \to \pp^1\) has infinitely many quadratic points pulled back under \(\pi\), this doesn't necessarily explain \textit{all} of the quadratic points on \(X\).  Indeed, a curve of genus \(2\) with simple Jacobian of positive rank will have infinitely many quadratic points \textit{not} pulled back from a map to a lower genus curve \cite[Remark 4.3.4]{VV}.  In another direction, there exist genus \(3\) curves that are simultaneously a double cover of \(\pp^1\) and a positive rank elliptic curve, and so they have infinitely many quadratic points that are not pulled back under \(\pi\).  Nevertheless, Vojta established inequalities involving the arithmetic discriminant which imply that if the genus of \(X\) is at least \(4\), then all but finitely many quadratic points are pulled back under \(\pi\) \cite[Corollary 0.5]{vojta_92}, \cite{Vojta-arithmeticdiscriminant}. More generally, Vojta showed that, assuming the genus of \( X \) is sufficiently large, a given map \( \mu \colon X \rightarrow \pp^{1} \) \defi{contracts} all but finitely many points of a fixed degree, i.e., \(k(\mu(x)) \subsetneq k(x)\) \cite[Corollary 0.3]{vojta_92}.

In this note, we generalize Vojta's results to higher dimensions.  Let \(\pi \colon X \to \pp^r\) be a ramified double cover.   
Vojta has conjectured a strong inequality (see Section~\ref{sec:vojta}) involving the canonical heights of algebraic points and the discriminants of the number field they generate, that, in particular, implies the Bombieri--Lang Conjecture.  We show that this also implies that \(\pi\) constrains the quadratic points on \(X\): once the branch divisor has degree \(2m\) with \(m\geq r+4\), the quadratic points on \(X\) \textit{not} lying in the fibers of \(\pi\) cannot be Zariski dense.
This is a special case of a more general result about degree \(d\) points on a cyclic cover \(\pi \colon X \to \pp^r\) that are not contracted by the map \(\pi\).

\begin{ithm}\label{thm:main}
    Assume Vojta's Conjecture (Conjecture~\ref{conj:vojta}).  Let \(\pi \colon X \to \pp^r\) be a degree \(e\) cyclic cover over a branch divisor of degree \(em\) with \(m > \frac{r + 2d-1}{e-1}\) such that \(X\) has canonical singularities.  Then
    \[\overline{\left\{x \in X: \deg(x) = d, \text{ and } k(\pi(x)) = k(x)\right\}}^{\text{Zar}} \neq X.\]
\end{ithm}

\begin{icor}\label{cor:small_gcd}
    Assume Vojta's Conjecture (Conjecture~\ref{conj:vojta}).  In the setup of Theorem~\ref{thm:main},
    \begin{enumerate}
        \item If \((d,e) = 1\), then the degree \(d\) points on \(X\) cannot be Zariski dense.
        \item If \((d,e) = p\) is prime, then there exists a Zariski closed \(V \subsetneq X\) such that for every degree \(d\) point \(x \in X \smallsetminus V\), the residue field \(k(x)\) is a cyclic degree \(p\) extension of \(k(\pi(x))\).
    \end{enumerate}
\end{icor}

To analyze the sharpness of the bounds in Theorem~\ref{thm:main}, we will focus on the case \(d=e=2\) of quadratic points on double covers of projective space. Here, Theorem~1 implies that outside a proper subset, every quadratic point on $X$ is pulled back along $\pi \colon X \to \pp^r$ from a $k$-point of $\pp^r$.
In the case \(\dim X =1\), the bound \(m \geq 5\) exactly translates into the genus of \(X\) being at least \(4\), and so Theorem~\ref{thm:main} is consistent with Vojta's Theorem. When \( \dim X = 2\), the bound becomes $m \geq 6$. In this case, we show that Theorem 1 is \textit{almost} optimal, both in terms of $\deg D$ and in terms of the singularities of $X$.

\begin{ithm}\label{thm:SharpExamples}
    Let $\pi \colon X \rightarrow \pp^{2}$ be a double cover branched over a divisor $D \subset \pp^{2}$ of degree $2m$.
    \begin{enumerate}
    \item\label{sharp1} For every $m \leq 4$, there is an example where $D$ is smooth and
        \[\overline{\left\{x \in X: \deg(x) = 2, \text{ and } k(\pi(x)) = k(x)\right\}}^{\text{Zar}} = X.\]
    \item\label{sharp2} There exist examples of singular $D$ with \(m\) arbitrarily large such that
        \[\overline{\left\{x \in X: \deg(x) = 2, \text{ and } k(\pi(x)) = k(x)\right\}}^{\text{Zar}} = X.\]
    As \(m \to \infty\), these surfaces $X$ are of general type with increasing volume in $m$. 
    \end{enumerate}
\end{ithm}

\noindent Theorem~\ref{thm:SharpExamples}\eqref{sharp2} also illustrates the importance of requiring that \(\pi \colon X \to \pp^r\) is a finite morphism, as opposed to simply a rational map or even a (non-flat) morphism.  Specifically, by resolving singularities in Theorem~\ref{thm:SharpExamples}\eqref{sharp2}, we give an infinite family of smooth surfaces \( \widetilde{X} \) with degree of irrationality equal to \(2\), unbounded volume \(K_{\widetilde{X}}^2\), and dense quadratic points not contracted by the morphism \(\widetilde{X} \to \pp^2\) of degree \(2\).

When \(r=d=e=2\), there is one interesting case that is not addressed by Theorems~\ref{thm:main} and \ref{thm:SharpExamples}. 

\begin{iquestion}\label{q:deg10}
    Does there exist a double cover $\pi \colon X \rightarrow \pp^{2}$ branched over a divisor of degree $10$ such that $X$ has canonical singularities (or is even smooth) for which
    \[\overline{\left\{x \in X: \deg(x) = 2, \text{ and } k(\pi(x)) = k(x)\right\}}^{\text{Zar}} = X?\]
\end{iquestion}

Our technique for producing examples in Theorem~\ref{thm:SharpExamples}\eqref{sharp1} is to find a double cover \(\pi \colon X \to \pp^2\) admitting \textit{another} degree \(2\) map \(f \colon X \rat V\) such that a Zariski dense set of fibers of \(f\) are quadratic points on \(X\).  We show in Theorem~\ref{thm:degree_2_maps} that, assuming the Bombieri--Lang conjecture, if Question~\ref{q:deg10} has an affirmative answer, then the Zariski dense set of quadratic points \textit{cannot} be fibers of a map of degree \(2\).  Thus, if Theorem~\ref{thm:main} is sharp when \(r=d=e=2\), the additional dense quadratic points will come from a more exotic source.

\subsection*{Acknowledgements}
This material is based partially upon work supported by National Science Foundation grant DMS-1928930 while the last two authors were in residence at the Simons Laufer Mathematical Sciences Institute in Berkeley, California, during the Spring 2023 semester on \textit{Diophantine Geometry}.  We thank SLMath for the environment and support that made this work possible. We also thank Bianca Viray for coorganizing the American Institute of Mathematics (AIM) workshop on \textit{Degree \(d\) points on surfaces}.  We thank AIM for creating a stimulating environment and all of the participants for helpful conversations, especially Niven Achenjang, Andres Fernandez Herrero, Yifeng Huang, Lena Ji, and Ritvik Ramkumar. B.C.\ thanks his advisor Ravi Vakil for helpful comments on Theorem~\ref{thm:degree_2_maps}. During the preparation of this article, N.C. was supported in part by an NSF postdoctoral fellowship DMS-2103099. B.C.\ was supported in part by a NSF Graduate Research Fellowship under grant DGE-2146755. H.P.\ was supported by ANID Fondecyt Regular grant 1230507 from Chile. I.V.\ was supported in part by NSF grants DMS-2200655 and DMS-2338345.

\section{Vojta's Conjecture}

\subsection{Background and notation on height functions}\label{sec:bakcground}
Let \(k\) be a number field and let \(M_k\) denote the set of places of \(k\) and \(S_k\) the set of archimedean places.  Given a place \(v \in M_k\), we write \(\|\cdot\|_v\) for the (normalized) absolute value at \(v\) as in \cite[Equation (1.1.2)]{vojta_diophantine}, where the normalization is such that the product formula 
\begin{equation}
\prod_{v \in M_k} \|x\|_v = 1
\end{equation} 
holds for all \(x \in k^\times\).  Similarly, given a nonarchimedean place \(v\) corresponding to a prime ideal \(\fp\), the absolute value of any nonzero fractional ideal \(\fa\) can be defined by \(\|\fa\|_v = N \fp^{-\ord_{\fp}\fa}\).  Observe that 
\begin{equation}\label{eq:def_Lk}
L_k(\fa) \colonequals \prod_{v \in M_k \smallsetminus S_k} \|\fa\|_v^{-1} = \prod_{\fp \mid \fa} N \fp^{\ord_{\fp}\fa} = |N_{k/\qq}\fa|.
\end{equation}

Given an extension \(L/k\) of number fields, and a place \(w \mid v\), for every \(x \in k\) we have \(\|x\|_w = \|x\|_v^{[L_w:k_v]}\).  Hence the terms in the product formula yield 
\[\prod_{{ w \in M_L \atop w \mid v }} \|x\|_w = \|x\|_v^{[L:k]}, \qquad  \text{for all }x \in k^\times.\]

The \defi{(logarithmic) height} of a point \(P = [x_0: \cdots : x_r] \in \pp^r(\bar{k})\) is defined by 
\[h(P) = \frac{1}{[k(P):k]} \log \left( \prod_{v \in M_k} \max_i\{\|x_i\|_v\} \right).\]
Given a very ample line bundle \(A\) on a variety \(X/k\), we write \(h_A\) for the height function on the points of \(X(\bar{k})\) given by restricting \(h\) to the image of \(X\) under the complete linear system of \(A\).  In this way, the height function \(h \colon \pp^r(\bar{k}) \to \rr\) defined above is associated to the line bundle \(\O_{\pp^r}(1)\) and we omit this from the notation for simplicity.  The height function \(h_A\) is well-defined up to a bounded function.  Since any line bundle can be written as the difference of very ample line bundles, this definition extends linearly to the height function associated to any line bundle.  See \cite[Chapter 1]{vojta_diophantine} for more details and properties.  Since it will be important later, recall that the height function associated to an effective divisor \(E\) is at least a bounded function:
\begin{equation}\label{eq:effective} h_E \geq O(1)\end{equation}
away from the points of \(E\) itself \cite[Proposition 1.2.9(e)]{vojta_diophantine}.

We write \(D_{L/k}\) for the relative discriminant ideal.  The \defi{(logarithmic) discriminant} of \(L/k\) is defined by
\[d_k(L) = \frac{\log|N_{k/\qq}D_{L/k}|}{[L:k]}.\]
Using multiplicativity of the discriminant in towers of extensions, this can also be expressed as
\[d_k(L) = \frac{\log|D_{L/\qq}|}{[L:k]} - \log |D_{k/\qq}| = [k:\qq] \left(d_{\qq}(L) - d_{\qq}(k)\right).\]
If \(X\) is a variety over \(k\) and \(P \in X\) is a closed point with residue field \(k(P)\), we define the (logarithmic) discriminant \(d_k(P)\) of \(P\) to be \(d_k(k(P))\).

In the proof of Theorem~\ref{thm:main}, we will need the following result of Silverman, which shows that the height of an algebraic point is bounded from below by its discriminant.

\begin{lem}[Silverman]\label{lem:silverman}
Let \(P \in \pp^r(\bar{k})\) with \([k(P):k] = d\).  Then
\[d_k(P) \leq (2d-2)h(P)   + O(1), \]
where the implicit constant depends on \(k\).
\end{lem}
\begin{proof}
By \cite[Theorem 2]{silverman}, we have
\[h(P) \geq \frac{1}{2d-2} \left( \frac{\log L_k(D_{k(P)/k})}{d} -\#S_k \log d \right).\]
Combining \eqref{eq:def_Lk} and the definition of \(d_k\) gives
\[\frac{\log L_k(D_{k(P)/k})}{d} = \frac{\log |N_{k/\qq} D_{k(P)/k}|}{d} = d_k(P),\]
and so Silverman's theorem rearranges to the desired statement.
\end{proof}

\subsection{Vojta's Conjecture}\label{sec:vojta}

Motivated by results in Nevanlinna theory, Vojta made the following conjecture, which (in particular) implies the Bombieri--Lang Conjecture. 

\begin{conj}[{\cite[Conjecture 5.2.6]{vojta_diophantine}, \cite[Conjecture 2.1]{vojta_abc}}]\label{conj:vojta}
Suppose that \(X\) is a smooth projective variety defined over a number field \(k\).  Let \(A\) be a big divisor on \(X\) and let \(\epsilon > 0\).  For any \(d >0\), there exists a proper Zariski closed subvariety \(Z = Z(d, X, A, k, \epsilon) \subsetneq X\) such that for all closed points \(P\in (X\smallsetminus Z)\) with \([k(P):k] \leq d\), we have
\[ h_{K_X}(P) \leq d_k(P) + \epsilon h_A(P) + O(1). \]
\end{conj}

\noindent Note that the original statement of this conjecture in \cite[Conjecture 5.2.6]{vojta_diophantine} involves a factor of \(\dim X\) multiplied by the discriminant.  Subsequently, this was removed in \cite[Conjecture 2.1]{vojta_abc}.

\section{Proof of Theorem~\ref{thm:main} and consequences}

Let \(\pi \colon X \to \pp^r\) be a cyclic cover of degree \(e\) branched over a divisor of degree \(em\) and assume that \(X\) has canonical singularities.  Let \(f\colon \widetilde{X} \to X\) denote a resolution of singularities.  Since \(X\) has canonical singularities, \(K_{\widetilde{X}} = f^*K_X  + E\), where \(E\) is effective.
By the Riemann--Hurwitz formula, we have
\[K_X = \pi^*(K_{\pp^r}\otimes \O_{\pp^r}((e-1)m)) \simeq \pi^*\O_{\pp^r}((e-1)m-r-1).\]
Putting these together, we have
\[K_{\widetilde{X}} = f^* \pi^* \O_{\pp^r}((e-1)m-r-1) + E.\]
Since \(E\) is effective, equation~\eqref{eq:effective} implies that for all \(P \in (\widetilde{X} \smallsetminus E) \), we have
\begin{equation}\label{eq:can_ht}
h_{K_{\widetilde{X}}}(P) \geq h_{f^* \pi^* \O_{\pp^r}((e-1)m-r-1)}(P) + O(1) = ((e-1)m-r-1)h(\pi(f(P))) + O(1),\end{equation}
where the equality comes from the linearity of the height in terms of the line bundle.

The canonical bundle \(K_{\widetilde{X}}\) is big using our hypothesis \(m > \frac{r+2d -1}{e-1}\).  Hence Conjecture~\ref{conj:vojta} on \(\widetilde{X}\) with \(A = K_{\widetilde{X}}\) and \(\epsilon\) sufficiently small implies that there exists a proper Zariski closed \(\widetilde{Z} = Z(d, \widetilde{X}, k, \epsilon)\) such that for all degree \(d\) points \(P \in (\widetilde{X} \smallsetminus \widetilde{Z}) \), we have
\begin{equation}\label{eq:vc}
(1 - \epsilon)h_{K_{\widetilde{X}}}(P) \leq d_k(P) + O(1).
\end{equation}
Combining \eqref{eq:can_ht} and \eqref{eq:vc} yields, for all degree \(d\) points \(P \in \widetilde{X} \smallsetminus (\widetilde{Z} \cup E) \),
\begin{equation}\label{eq:vc2}
    (1 - \epsilon)((e-1)m-r-1)h(\pi(f(P))) \leq d_k(P) + O(1).
\end{equation}

Abusing notation, suppose that \(P\) is now a degree \(d\) point on the smooth locus of \(X\) that is not contracted by \(\pi\), i.e., \(k(P)=k(\pi(P))\), and not in the image of \(\widetilde{Z} \subset \widetilde{X}\). Since \(P\) is contained in the locus where \(X\) and \(\widetilde{X}\) are isomorphic, it must satisfy equation~\eqref{eq:vc2} and 
\(d_k(P) = d_k(\pi(P))\):
\[(1-\epsilon)((e-1)m-r-1)h(\pi(P)) \leq  d_k(\pi(P)) + O(1).\]
Combining this with Lemma~\ref{lem:silverman}, we obtain
\begin{equation}\label{eq1}(1-\epsilon)((e-1)m-r-1) h(\pi(P)) \leq (2d-2) h(\pi(P)) + O(1).\end{equation}
If \(m > \frac{r + 2d -1}{e-1}\) then, for \(\epsilon\) sufficiently small, \((1-\epsilon)((e-1)m-r-1) - 2(d-1)\) is positive and hence \(h(\pi(P))\) is bounded.  Thus there are only finitely many possible degree \(d\) points \(\pi(P) \in \pp^r\) for which \eqref{eq1} holds.  Since \(\pi\) is a finite map, there are only finitely many degree \(d\) points \(P \in X \smallsetminus f(\widetilde{Z} \cup E)\) that are not contracted by \(\pi\).  Thus the Zariski closure of the degree \(d\) points not contracted by \(\pi\) is contained in \(f(\widetilde{Z} \cup E)\) union the support of these finitely many additional exceptional points, and hence cannot equal \(X\). \qed

\begin{proof}[Proof of Corollary~\ref{cor:small_gcd}]
    Since \(\pi\) is a cyclic (Galois) cover of degree \(e\), for any point \(P \in \pp^r\), and any point \(Q\) in the fiber \(\pi^{-1}(P)\), the extension \(k(Q)/k(P)\) is cyclic of degree dividing \(e\). On the other hand, given a degree \(d\) point \(x \in X\), since \(k(\pi(x))\) is a subfield of \(k(x)\), we have that \([k(x): k(\pi(x))]\) divides \(d\).  Thus \([k(x): k(\pi(x))]\) divides the greatest common divisor \((d, e)\). 

    If \((d,e) = 1\), then this implies that \([k(x): k(\pi(x))] = 1\), and so a degree \(d\) point \(x\) can never be contracted.  Thus, under the hypotheses of Theorem~\ref{thm:main}, the degree \(d\) points are not Zariski dense.

    If \((d,e)=p\) is prime, then this implies that the only way a degree \(d\) point can be contracted is if \([k(x):k(\pi(x))] = p\) and, in fact, if \(k(x)/k(\pi(x))\) is a cyclic degree \(p\) extension.  Thus, by Theorem~\ref{thm:main}, all degree \(d\) points outside of a proper Zariski closed subset have this shape.
\end{proof}

\section{Examples}

We will prove the first part of Theorem~\ref{thm:SharpExamples} by a series of examples.  
In the odd cases (\(m=1,3\)), the particular equations of the surface are not crucial and we give a general recipe to construct examples.  In the even cases (\(m=2,4\)), we employ a more subtle construction and we will rely on a particular choice of surface.  In all of these examples, we use the Hilbert Irreducibility Theorem to guarantee the existence of dense quadratic points (see \cite[Section 3.3]{VV}).

\begin{example}[\boldmath\(m=1\)]\label{ex:meq1}
Let \(X \subset \pp^3_\qq\) be a smooth quadric surface.  Choose two distinct points \(p_1 \neq p_2 \in \pp^3(\qq) \smallsetminus X(\qq)\) and let \(\pi_i \colon Q \to \pp^2\) denote projection from \(p_i\).  Each map \(\pi_i\) is a double cover branched over a smooth plane conic.  By the Hilbert Irreducibility Theorem, \(X\) has dense quadratic points coming from the fibers of both \(\pi_1\) and \(\pi_2\) which do not agree, aside from the one common fiber \(\operatorname{span}(p_1, p_2) \cap X\).\hfill\(\righthalfcup\)
\end{example}

\begin{example}[\boldmath\(m=3\)]\label{ex:meq3}
    Let \(X \subset \pp^2 \times \pp^2\) be a smooth complete intersection surface of type \((1,1), (2,2)\).  By the adjunction formula, \(X\) is a K3 surface admitting two distinct degree \(2\) maps \(\pi_1 \colon X \to \pp^1\) and \(\pi_2 \colon X \to \pp^2\), which are necessarily branched over smooth plane sextic curves.  By the Hilbert Irreducibility Theorem, there are dense quadratic points on \(X\) contained in the fibers of both \(\pi_1\) and \(\pi_2\) individually.\hfill\(\righthalfcup\)
\end{example}

In order to motivate the constructions for \(m=2,4\), we briefly recall how to construct examples showing that Vojta's result on quadratic points on hyperelliptic curves of genus at least \(4\) is sharp.  Let \(E\) be a positive rank elliptic curve with Weierstrass model \(y^2 = f(x)\).  Then the genus \(3\) curve \(C\) with equation \(y^2 = f(x^2)\) is a double cover of \(E\) via the map \((x, y) \mapsto (x^2, y)\), and hence has dense quadratic points whose \(x\)-coordinate is not rational; such points are not contracted by the hyperelliptic map.  Geometrically, the curve \(C \to \pp^1\) is the base-change of \(E \to \pp^1\) via the degree \(2\) map \(\pp^1 \to \pp^1\) given by \(x \mapsto x^2\), which is the quotient of \(\pp^2\) by the involution \([x:y] \mapsto [-x:y]\).  

\renewcommand{\P}{\mathbb{P}}

We will adapt this basic strategy to the case of surfaces.
Consider a double cover $\pi \colon X \rightarrow \pp^{2}$ branched over a divisor $D = \{ s = 0 \} \subset \pp^{2}$ of degree $2m$. Suppose there exists an involution $\iota \colon \pp^{2} \rightarrow \pp^{2}$ such that the section $s \in H^{0}(\pp^{2}, \O(D))$ is invariant, i.e., $\iota^{\ast}s = s$. The involution $\iota$ on $\P^{2}$ then lifts to an involution $\tau$ of $X$, and if the quotient $Y \colonequals X/\tau$ contains dense rational points then this gives a different source of quadratic points than $\pi$.
For $m = 4$, we will need the following lemma.

\begin{lem}\label{lem:elliptic}
The curve \( \{ Y^2 = X^4 + (t^8-1)Z^4 \} \subset \pp(1, 1, 2)\) over \(\qq(t)\) has dense \(\qq(t)\)-points.
\end{lem}
\begin{proof}
    Two such points are \(P \colonequals [1:1:0]\) and \(Q_t\colonequals[1:t^4:1]\).  It suffices to show that on the elliptic curve with origin \(P\), the point \(Q_t\) is non-torsion.  Furthermore, it suffices to find one specialization \(t_0 \in \qq\) such that \(Q_{t_0}\) is non-torsion.  Magma \cite{magma} verifies this for \(t_0 = 2\) via \texttt{C := HyperellipticCurve([255, 0, 0, 0, 1]); P := C![0,1];} \texttt{E, phi := EllipticCurve(C,P); Order(phi(C![1,16]));} which returns \texttt{0}, indicating that this point has infinite order.
\end{proof}

\begin{example}[\boldmath\(m=2,4\)]\label{ex:meq4}
Let $\pi \colon X \rightarrow \P^{2}_{\qq}$ be the double cover branched over the divisor $D = \{ s = 0 \}$, where $s(x,y,z) \colonequals x^{2m} - y^{2m} + z^{2m} \in H^{0}(\O_{\P^{2}}(2m))$.
Consider the involution on $\P^{2}$ defined by
\(\iota \colon [x:y:z] \mapsto [x:y:-z]\).
The quotient $\pp^{2}/\iota$ is isomorphic to $\pp(1,1,2)$ via the map
\[ g \colon \pp^{2} \rightarrow \pp(1,1,2), \quad [x:y:z] \mapsto [x:y:z^2 = u], \]
where $u$ has weight 2. 
By construction, the section \(s\) is invariant under the involution \(\iota\) and so it descends to the quotient \(\pp^2/\iota\);
in the above variables, $g(D)$ has equation $x^{2m} + y^{2m} + u^m = 0$.  Let \(Y\) be the double cover of \(\pp^2\) branched over \(g(D)\):
\[Y = \{v^2 = x^{2m} + y^{2m} + u^m\} \subset \pp(1,1,2,m) \ni [x:y:u:v].\]
Since $\pp(1,1,2)$ has an $A_{1}$-singularity at the cone point \([0:0:1]\), which is disjoint from $g(D)$, the surface $Y$ inherits two $A_{1}$-singularities.  Blowing up these singular points yields a resolution \(\widetilde{Y} \rightarrow Y\) that is the double cover of the Hirzebruch surface \(\ff_2\) (the blowup \(\pp(1,1,2)\) at the cone point) over (the isomorphic preimage of) \(g(D)\). 

(Although this is not strictly necessary to the argument, note that, in terms of the classes \(F\), a fiber of \(\ff_2 \to \pp^1\), and \(S\), the exceptional divisor of \(\ff_2 \to \pp(1,1,2)\), we have
\(K_{\ff_2} = -2S - 4F\) and \([g(D)] = mS + 2mF\), so \(\widetilde{Y}\) is a surface of Kodiara dimension \(-\infty\) when \(m=2\) and a K3 surface when \(m=4\).  In fact, we will crucially show below that \(\widetilde{Y}\) is an \textit{elliptic} K3 surface when \(m=4\).)

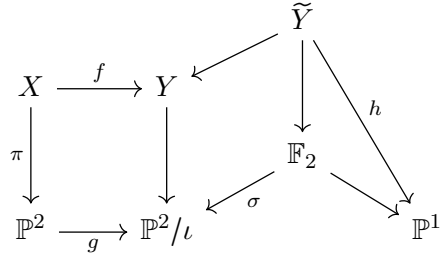
\begin{center}
\begin{tikzcd}[row sep=small]
    & & \widetilde{Y} \arrow[dd] \arrow[ld] \arrow[dddr, "h"] \\
    X \arrow[dd, swap, "\pi"] \arrow[r, "f"] & Y \arrow[dd] & \\
    & & \mathbb{F}_{2} \arrow[ld, "\sigma"] \arrow[dr] \\
    \pp^{2} \arrow[r, swap, "g"] & \pp^{2}/\iota & & \pp^1
\end{tikzcd}
\end{center}

Explicitly, the projective bundle map \(\ff_2 \to \pp^1\) resolves the rational map \(\pp(1,1,2) \dashrightarrow \pp^1\) given by \([x:y:u] \mapsto [x:y]\).  

When \(m=2\), the generic fiber of \(h \colon \widetilde{Y} \to \pp^1\) has affine equation \(v^2 = u^2 + t^4 - 1\) over \(\qq(t)\), where \(t = x/y\).  This is a conic with a \(\qq(t)\)-rational point, and hence a dense set of \(\qq(t)\)-rational points.
When \(m=4\), the generic fiber of \(h \colon \widetilde{Y} \to \pp^1\) has affine equation \(v^2 = u^4 + t^8 - 1\) over \(\qq(t)\), where \(t = x/y\).  This curve has dense \(\qq(t)\)-points by Lemma~\ref{lem:elliptic}. In both cases, the union of sections of \(h \colon \widetilde{Y} \to \pp^1\) is dense in \(\widetilde{Y}\).  

In both cases, passing to the birational surface \(Y\), the union of the rational curves on \(Y\) is also dense.  The preimage any such rational curve under the degree \(2\) map \(f \colon X \to Y\) is a curve with dense quadratic points  contained in the fibers of \(f\) by Hilbert's Irreducibility Theorem.  In particular, \(X\) has dense quadratic points in the fibers of \(f\) (and not in the fibers of \(\pi\), since \(f\) and \(\pi\) are distinct degree \(2\) morphisms.) Explicitly, these quadratic points have \textit{nonrational} \(z\)-coordinate, and so they map to quadratic points on \(\pp^2\) via $\pi$.\hfill\(\righthalfcup\)
\end{example}

Next, we will give a family of examples that demonstrates how Theorem~\ref{thm:main} can fail when the branch divisor of $\pi$ is too singular. This will prove the second part of Theorem~\ref{thm:SharpExamples}.

\begin{example}
To set notation in this example, given a Hirzebruch surface \(\ff_n = \pp(\O_{\pp^1} \oplus \O_{\pp^1}(n))\), with \(n > 0\), we will write \(F\) for the class of a fiber of the projective bundle \(\ff_n \to \pp^1\) and \(S_0\) for the class of the unique section of negative self-intersection \(-n\).  These classes generate \(\Pic(\ff_n)\) and we write \(S_\infty = S_0 + n F\) for the class of a complementary section with self-intersection \(n\).  In terms of these classes, we have \(K_{\ff_n} = -2S_0 - (n+2)F = -2S_\infty + (n-2)F\).

Let \(Y \to \ff_{2n}\) be a double cover branched over a smooth curve in class \(4S_\infty\) such that \(Y\) contains a dense set of rational curves over \(\qq\); an explicit construction of such a surface \(Y\) is the pullback of \(\widetilde{Y} \to \ff_2\) from the \(m=4\) case of Example~\ref{ex:meq4} under a map \(\gamma \colon \pp^1 \to \pp^1\) of degree \(n\):
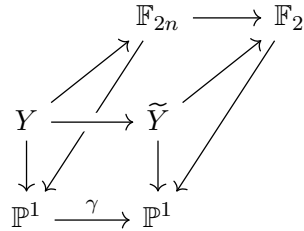
\begin{center}
\begin{tikzcd}
& \ff_{2n} \arrow[r] \arrow[ldd] & \ff_{2} \arrow[ldd] \\
Y \arrow[r, crossing over] \arrow[d] \arrow[ru] & \widetilde{Y} \arrow[ru] \arrow[d] & \\
\mathbb{P}^1 \arrow[r, "\gamma"] & \mathbb{P}^1
\end{tikzcd}
\end{center}

Let \(\tau \colon \ff_n \to \ff_n\) be the involution induced by the map of vector bundles
\begin{equation}\label{eq:tau_vb}
\mathcal{O}_{\pp^1} \oplus \mathcal{O}_{\pp^1}(n) \xrightarrow{\begin{pmatrix}
1 & 0 \\
0 & -1
\end{pmatrix}} \mathcal{O}_{\pp^1} \oplus \mathcal{O}_{\pp^1}(n).
\end{equation}

The fixed locus of $\tau$ consists of two disjoint sections in classes $S_{0}$ and $S_{\infty}$. 
Since the invariants of the vector bundle map \eqref{eq:tau_vb} involve squaring a section of \(\O(n)\), after projectivizing, the quotient map is
\[
g \colon \mathbb{F}_n \to \mathbb{F}_{2n}
\]
which is branched over the union of two disjoint sections in classes \(S_0\) and \(S_\infty\) on \(\ff_{2n}\). 
Pulling back \(Y \to \ff_{2n}\) along $g$, we obtain a double cover \(\wt{\pi} \colon \wt{X} \to \ff_2\) in the following diagram:
\begin{center}
\begin{tikzcd}
\wt{X} \arrow[r] \arrow[d, "\wt{\pi}"] & Y \arrow[d] \\
\ff_n \arrow[r, "g"] & \ff_{2n}
\end{tikzcd}
\end{center}
The surface \(\wt{X}\) has dense quadratic points in the fibers of both maps \(\wt{\pi} \colon \wt{X} \to \ff_n\) and \(\wt{X} \to Y\) by the Hilbert Irreducibility Theorem.

We now show that invariants of \(\wt{X}\) are unbounded as functions of \(n\).  Recall that \(Y \to \ff_{2n}\) is branched over a smooth curve in class \(4S_\infty\).  The class \(S_\infty\) on \(\ff_{2n}\) pulls back to \(2S_\infty\) on \(\ff_n\) under \(g\) since a section in class \(S_\infty\) is in the branch locus of \(g\).  Thus the cover \(\wt{X} \to \ff_n\) is branched over a curve \(C\) in class \(8S_\infty\).  We therefore have
\[
K_{\wt{X}} = {\wt{\pi}}^*\left(K_{\mathbb{F}_n} + 4S_\infty\right) = {\wt{\pi}}^*\left(2S_\infty + (n-2)F\right) = {\wt{\pi}}^*\left(2S_0 + (3n-2)F\right)
\]
which is ample for $n > 2$ since \(2S_0 + (3n-2)F\) on \(\ff_n\) is ample for \(n > 2\)  \cite[Chapter V, Corollary 2.18]{hartshorne}. This has volume
\(K_{\wt{X}}^2 = 2 (2S_\infty + (n-2)F)^2
= 16(n-1),
\)
which is unbounded as $n \to \infty$.  Furthermore, by the adjuction formula, the genus of the branch curve \(C\) of \(\wt{\pi}\) is 
\[
\frac{8S_\infty \cdot (K_{\mathbb{F}_n} + 8S_\infty)}{2} + 1 = 4S_\infty \cdot (6S_\infty + (n-2)F) + 1 = 28n - 7,
\]
which is, again, unbounded as \(n \to \infty\).

Finally, we show that \(\wt{X}\) is birational to a variety \(X\) -- which does \textit{not} have canonical singularities --  that admits a double cover \(\pi \colon X \to \pp^2\) branched over a divisor of unbounded degree as \(n \to \infty\).  Let \(\beta_1 \colon \widetilde{\ff} \to \ff_n\) be the blowup of \(n-1\) points \(p_1, \dots, p_n\) in distinct fibers \(F_1, \dots, F_{n-1}\) of \(\ff_n \to \pp^1\) and not lying on the branch curve \(C\) or on the distinguished negative section \(S_0\).  The proper transforms \(\widetilde{F}_1, \dots, \widetilde{F}_{n-1}\) are \((-1)\)-curves on \(\widetilde{\ff}\) and so can be contracted to yield another projective bundle over \(\pp^1\), which is an elementary modification of \(\ff_n\) \cite[Chapter V, Example 5.7.1]{hartshorne}.  By the push-pull formula, the image of the proper transform of \(S_0\) is a section of self-intersection \(-1\) and so the projective bundle is isomorphic to \(\ff_1\).  Write \(\beta_2\colon \widetilde{\ff} \to \ff_1\) for this blow-down morphism.  Together these yield a birational map \(\psi \colon \ff_n \rat \ff_1\) defined away from \(p_1, \dots, p_n \in \ff_n\), and hence defined on \(C\).  Since the fibers \(F_1, \dots, F_{n-1}\) are contracted by this map, and \(C\) meets each of \(F_1, \dots, F_{n-1}\) with multiplicity \(8\), the image \(\psi(C) \subset \ff_1\) has \(n-1\) \(8\)-fold points along the negative section \(S_0 \subset \ff_1\).
Contracting the \(-1\)-curve \(S_0 \subset \ff_1\) yields the morphism \(\ff_1 \to \pp^2\).  Write \(D\) for the image of \(\psi(C)\) under this morphism, which is a curve of degree \(8n\) and has a point of multiplicity \(8(n-1)\) at the image of \(S_0\).  Let \(\pi \colon X \to \pp^2\) be the double cover branched over \(D\) fitting into the diagram:

\begin{center}
    \begin{tikzcd}
        \wt{X} \arrow[dd, "\wt{\pi}", swap]\arrow[rrr, dashed, "\text{bir}"] &&& X \arrow[dd, "\pi"] \\
        &\widetilde{\ff} \arrow[dl, "\beta_1", swap] \arrow[dr, "\beta_2"]\\
        \ff_n \arrow[rr, dashed, "\text{bir}"] && \ff_1 \arrow[r] & \pp^2
    \end{tikzcd}
\end{center}

Since \(X\) is birational to \(\wt{X}\), the quadratic points not in the fibers of \(\pi\) are Zariski dense.  The branch curve \(D\) has degree \(8n\), but because of the multiplicity \(8(n-1)\) point, the double cover \(X\) does not have canonical singularities.\hfill\(\righthalfcup\)
\end{example}

Finally, as a consequence of the following theorem, assuming the Bombieri--Lang Conjecture, similar constructions as in Examples~\ref{ex:meq1}-\ref{ex:meq4} of double covers \(\pi \colon X \to \pp^2\) with canonical singularities that are simultaneously double covers of an auxiliary variety \(V\) with dense rational points cannot exist if the degree of the branch divisor of \(\pi\) is at least \(10\).  Hence, assuming the Bombieri--Lang Conjecture, if there is a positive answer to Question~\ref{q:deg10}, the second source of quadratic points \textit{cannot} come from a second double cover.

\begin{thm} \label{thm:degree_2_maps}
Suppose that $\pi \colon X \to \PP^r$ is a double cover branched over a divisor of degree $2m$ with $m \ge r + 2$ such that $X$ has canonical singularities. If $X \rat V$ is any degree $2$ rational map to an $r$-fold \(V\), then $X \dashrightarrow V$ is birational to one of the following:
\begin{enumerate}
    \item the double cover $\pi : X \to \P^r$
    \item the morphism fitting into a diagram
    \begin{center}
    \begin{tikzcd}
    X \arrow[r] \arrow[d] & V \arrow[d]
    \\
    \P^r \arrow[r] & \P^r / \iota 
    \end{tikzcd}
    \end{center}
    where $\iota$ is a linear involution of $\P^r$.
\end{enumerate}
If, furthermore, \(r \geq 2\) and \(m \geq r+3\) then \(V\) is of general type, and so, assuming the Bombieri--Lang Conjecture, the rational points on \(V\) are not Zariski dense.
\end{thm}

\begin{proof}
The map $f \colon X \rat V$ is degree $2$ and therefore induces on $X$ a rational involution $\wt{\iota}$. However, $X$ has canonical bundle $K_X = \pi^*(K_{\P^r} + m H) = \pi^* (m - r-1)H$ which is ample for $m \ge r + 2$. Therefore, $X$ is embedded by some pluri-canonical linear series $|nK_X|$ for a large power $n > 0$. Since $\wt{\iota}$ acts on $H^0(X, \cO_X(r K_X))$ and is regular on the canonical model, this implies that $\wt{\iota} \in \Aut(X)$ is a regular involution. Furthermore, it preserves the canonical linear series $|K_X|$. 
Since $\pi_* \omega_X = \pi_* \pi^* \cO_{\pp^r}(m-r-1) = \cO_{\pp^r}(m-r-1) \ot \pi_* \cO_X = \cO_{\pp^r}(m-r-1) \oplus \cO_{\pp^r}(-r-1)$, we have that $\omega_X$ is globally generated and all its sections are pulled back from sections of $\cO_{\pp^r}(m-r-1)$.
Hence $|K_X|$ induces the map $X \xrightarrow{\pi} \P^r \embed \P^{N}$ (up to a linear change of variables on $\P^r$) where the second map is a Veronese embedding of degree $(m - r - 1)$. 
Therefore, $\wt{\iota}$ commutes with $\tau$ and induces a (regular, i.e., linear) automorphism of $\P^N$, which then induces a (regular, i.e., linear) automorphism $\iota$ of $\P^r$. These fit into the diagram above, which completes the first part of the proof.

We now show that if $m \ge r + 3$, then  $V$ is of general type. The divisorial components \(\Fix(\wt{\iota})\) of the fixed locus of $\wt{\iota}$ is contained in $\pi^* H$, where \(H\) is the hyperplane class on \(\pp^r\), which is the divisorial component of the fixed locus of the linear involution $\iota$. Precisely, there are two possibilities, corresponding to the trivial lift $\wt{\iota}$ with $\Fix(\wt{\iota}) = \pi^* H$ or $\tau \circ \wt{\iota}$ with fixed locus $\Fix(\wt{\iota}) =(\pi^* H) \cap R$ where $R \subset X$ is the ramification divisor. In either case, $f^* K_V = K_X - \mathrm{Fix}(\wt{\iota}) = \pi^* (m - r - 1) - \mathrm{Fix}(\wt{\iota})$ is ample for $m \ge r + 3$ and hence \(K_V\) is also ample. Therefore, assuming the Bombieri--Lang Conjecture, the quadratic points on \(X\) in the fibers of \(f\) cannot be Zariski dense, since their images are rational points on \(V\), which are not Zariski dense. 
\end{proof}

\begin{rem}
The analysis at the end of the proof of Theorem~\ref{thm:degree_2_maps} is different when $r = 1$. Indeed, then $\mathrm{Fix}(\iota) \sim 2 H$, since it consists of two reduced points, so when $m = r + 2 = 3$ one obtains a genus $3$ hyperelliptic curve with a $2:1$ map to an elliptic curve from construction (2). As discussed previously, this example shows Vojta's bounds are sharp in dimension $1$. Since in higher dimensions, the fixed locus of $\iota$ consists of a disjoint point and line, its divisorial component $H$ is of degree $1$. 
\end{rem}

\end{document}